\documentclass{amsart}
\usepackage{graphicx}
\usepackage{amsmath}% amstex ?
\usepackage{amssymb,bbm}%
\usepackage[numbers, square]{natbib}
\usepackage{graphicx}
\usepackage{color}

\newcommand{\Vol}{\mathop{\mathrm{Vol}}\nolimits}

\renewcommand{\Im}{\operatorname{Im}}
\renewcommand{\Re}{\operatorname{Re}}

\theoremstyle{plain}
\newtheorem{theorem}{Theorem}[section]
\newtheorem{lemma}[theorem]{Lemma}

\newtheorem{proposition}[theorem]{Proposition}

\theoremstyle{definition}

\theoremstyle{remark}

\begin{document}

\author[F. G\"otze]{Friedrich G\"otze}
\address{Friedrich G\"otze, Department of Mathematics,
Bielefeld University,
P. O. Box 10 01 31,
33501 Bielefeld, Germany}
\email{goetze@math.uni-bielefeld.de}

\author[D. Kaliada]{Dzianis Kaliada}
\address{Dzianis Kaliada, Institute of Mathematics, National Academy of Sciences of Belarus, 220072 Minsk, Belarus}
\email{koledad@rambler.ru}

\author[D. Zaporozhets]{Dmitry Zaporozhets}
\address{Dmitry Zaporozhets\\
St.\ Petersburg Department of
Steklov Institute of Mathematics,
Fontanka~27,
 191011 St.\ Petersburg,
Russia}
\email{zap1979@gmail.com}

%\shortauthors{F. G\"otze, D. Kaliada, D. Zaporozhets}

\title{Distribution of complex algebraic numbers}
\thanks{Supported by CRC 701, Bielefeld University (Germany).}

\begin{abstract}
For a region $\Omega \subset\mathbb{C}$ denote by $\Psi(Q;\Omega)$ the number of complex algebraic numbers in $\Omega$ of degree $\leq n$ and naive height $\leq Q$. We show that
$$
\Psi(Q;\Omega)=\frac{Q^{n+1}}{2\zeta(n+1)}\int_\Omega\psi(z)\,\nu(dz)+O\left(Q^n \right),\quad Q\to\infty,
$$
where $\nu$ is the Lebesgue measure on the complex plane and the function $\psi$ will be given explicitly.
\end{abstract}

\maketitle

\section{Introduction}

Results on the distribution of both the real and complex algebraic numbers concerning regular systems (Beresnevich \cite{vB99}, Bernik and Vasil'ev~\cite{BV99}; see also the review by Bugeaud~\cite{yB04}) suggest that for any fixed degree $n$ algebraic numbers of sufficiently large height are distributed quite regularly.

An important question for all algebraic numbers (of a given height) in this respect had been asked by Mahler in his letter to Sprin\-d\v{z}uk in 1985: what is the distribution of algebraic numbers of a fixed degree $n\geq 2$?

The following answer to this question was suggested in~\cite{dK14} (see also~\cite{dK13}, \cite{dK15} for the case $n=2$). Fix $n\geq2$ and consider an arbitrary interval $I\subset\mathbb{R}$. Denote by $\Phi(Q;I)$ the number of real algebraic numbers in $I$ of degree at most $n$ and height at most $Q$. Then
\begin{equation}\label{1817}
\Phi(Q;I)=\frac{Q^{n+1}}{2\zeta(n+1)}\int_I\varphi(x)\,dx+ O\left(Q^n \log^{l(n)} Q\right),\quad Q\to \infty,
\end{equation}
where $\zeta(\cdot)$  denotes  the Riemann zeta function and $l(n)$ is defined by
\[
l(n) =  \begin{cases}
   1, & n=2,\\
   0,  & n\geq 3.
 \end{cases}
\]
The limit density $\varphi$ is given by the formula
\[
\varphi(x)=\int_{B_x}\left|\sum_{k=1}^n kt_kx^{k-1}\right|\,dt_1\dots dt_n,
\]
where the domain $B_x$ is defined by
\[
B_x=\left\{(t_1,\dots,t_n)\in\mathbb{R}^n\,:\,\max_{1\leq k\leq n}|t_k|\leq1, \ |t_nx^n+\dots+t_1x|\leq1 \right\}.
\]
If  $x\in [-1+1/\sqrt2, 1-1/\sqrt2]$, then $\varphi(x)$ can be simplified as follows:
\[
\varphi(x)=\frac{2^{n-1}}{3}\left(3+\sum_{k=1}^{n-1}(k+1)^2x^{2k}\right).
\]
In \cite{GKZ15c} this result  was generalized to the limit correlations between real algebraic conjugates.
Note in passing that $2^{-n-1}\varphi$ coincides with the density of the real roots of a random polynomial with independent coefficients uniformly distributed on $[-1,1]$ (see, e.g.,~\cite[Section~3]{GKZ15b}).

\bigskip

The aim of this note is to obtain a \emph{complex} counterpart of~\eqref{1817}.
The real and complex cases are quite different from each other. In particular, the result for the non-real numbers can not be deduced from the real case.
Therefore we use a different approach to solve this problem.

\bigskip

Let us give a very brief overview of some related works. There are numerous papers studying the distribution of distances between algebraic conjugates.
In this active research area, notable results were obtained in~\cite{BM04}, \cite{jE04}, \cite{BBG10}, \cite{BM10}, \cite{yBaD11}, \cite{yBaD14}.
{This problem is closely related with the problem of the distribution of polynomial discriminants \cite{nBfG13}, \cite{fGdZ14}.}

In~\cite{BBG10} Beresnevich, Bernik and G\"otze obtained the following result. Let $n\ge 2$ and $0 <\rho \le \frac{n+1}{3}$. Then for all sufficiently large $Q$ and any interval $I\subset[-\frac12, \frac12]$ there exist at least $\frac12 Q^{n+1-2\rho}|I|$ real algebraic numbers $\alpha$ of degree $n$ and height $H(\alpha)\asymp_n Q$ having a real conjugate $\alpha^*$ such that $|\alpha - \alpha^*| \asymp_n Q^{-\rho}$.
%Note that this estimate corresponds with our results about the distribution of non-real algebraic numbers near the real line (see Propositions \ref{prop-repuls}).

Using potential theory Pritsker \cite{iP11} considered the case when $n\to\infty$ and found the asymptotic distribution of the roots of an integral polynomial whose generalized  Mahler measure satisfies some conditions. As a corollary he obtained the solution of Schur's problem on traces of algebraic numbers.
The paper \cite{iP11} also contains a number of references on this subject.
Pritsker's results are closely related to the problem of the distribution of the complex roots of random polynomials with i.i.d. coefficients  when $n\to\infty$. The landmark result of Erd\H{o}s and Tur\'{a}n~\cite{ET50} implies that the arguments of the complex roots  are asymptotically uniformly distributed (see~\cite{IZ13} for the proof without any additional assumption). Moreover, under some quite general assumptions the roots are clustered near the unit circle (see~\cite{SS62}, \cite{IZ13}).

Some papers are devoted to the asymptotic behavior of the number of algebraic elements $\alpha$ of a fixed degree $n$ and a bounded multiplicative Weil height $\mathcal{H}(\alpha)\le X$ over some base number field (as $X$ tends to infinity). Let $\overline{\mathbb{Q}}_n(X)$ denote the number of such elements over the field of rational numbers $\mathbb{Q}$.
Masser and Vaaler \cite{MV08} established the following asymptotic formula using a result by Chern and Vaaler~\cite{CV01}:
\[
\overline{\mathbb{Q}}_n(X) = \sigma_n X^{n(n+1)} + O\left(X^{n^2} \log^{l(n)} X\right),\quad X\to\infty,
\]
where the explicit factor $\sigma_n$ and the implicit big-O-notation constant depend on $n$ only.
Here the Weil height $\mathcal{H}(\alpha)$ can be expressed in terms of the Mahler measure by $\mathcal{H}(\alpha) = M(\alpha)^{1/n}$.
Note that $X$ is of order $Q^{1/n}$, where $Q$ is the upper bound for the corresponding naive heights.
In \cite{MV07} Masser and Vaaler generalized this result to arbitrary base number fields.
References and some historical results related to the topic can be found in \cite[Chapter 3, \S 5]{sL83}.
Note that these results are based on the use of the Weil height and do not overlap with ours.

\section{Main result}

For an integral polynomial
\[
p(x) = a_n x^n + \dots + a_1 x + a_0, \quad a_n\ne 0,
\]
its height is defined as $H(p):=\max_{0\le i\le n} |a_i|$.

A \emph{minimal polynomial} $p$ of an algebraic number $\alpha$ is an integral nonzero polynomial of the minimal degree with coprime coefficients such that $p(\alpha)=0$. Given an algebraic number $\alpha$, its  degree $\deg(\alpha)$ and height $H(\alpha)$ are defined as degree and height of the corresponding minimal polynomial.

We always assume that degree  $n$ is arbitrary but \emph{fixed}. Hence the constants in different asymptotic relations (as $Q\to\infty$)  in this paper might depend on $n$.

For a complex region $\Omega\subset\mathbb{C}$ denote by $\Psi(Q;\Omega)$ the number of algebraic numbers in $\Omega$ of degree at most $n$ and height at most $Q$. We always assume that $\Omega$ does not intersect the real axis and that its boundary consists of a finite number of algebraic curves.

\begin{theorem}\label{1844}
Let $n\ge 2$ be a fixed arbitrary integer.
We have that
\begin{equation}\label{1142}
\Psi(Q;\Omega)=\frac{Q^{n+1}}{2\zeta(n+1)}\int_\Omega\psi(z)\nu(dz)+O\left(Q^n \right),\quad Q\to\infty,
\end{equation}
where $\nu$ is the Lebesgue measure on the complex plane. The limit density $\psi$ is given by the formula
\begin{equation}\label{1132}
\psi(z)=\frac{1}{|\Im z|}\int_{D_z}\left|\sum_{k=1}^{n-1}t_k\left((k+1)z^{k}-\frac{\Im z^{k+1}}{\Im z}\right)\right|^2\,dt_1\dots dt_{n-1}.
\end{equation}
The integration is performed over the region
\begin{multline*}
D_z = \left\{(t_1,\dots,t_{n-1})\in\mathbb{R}^{n-1} : \max\limits_{1\le k\le n-1} |t_k|\le 1, \phantom{\left|\sum_{1}^1\right|} \right. \\
\left.
\left|z \sum_{k=1}^{n-1} t_k \left(z^{k} - \frac{\Im z^{k+1}}{\Im z}\right)\right|\le 1, \
\left|\frac{1}{\Im z}\sum_{k=1}^{n-1} t_k \Im z^{k+1}\right|\le 1 \right\}.
\end{multline*}
\end{theorem}
The implicit constant in the big-O-notation in \eqref{1142} depends on $n$, the number of the algebraic curves that form the boundary $\partial \Omega$, and their maximal degree only.

The proof of Theorem~\ref{1844} is given in Section~\ref{1845}. Now let us derive several properties of the limit density $\psi$.
\begin{proposition}
The function $\psi$ is positive on $\mathbb{C}$ and satisfies the following functional equations:
\begin{gather}
\psi(-z) = \psi(\bar z) = \psi(z), \label{eq-sym}\notag\\
\psi \left(\frac{1}{z}\right) = |z|^4 \psi(z). \label{1436}
\end{gather}
\end{proposition}
\begin{proof}
The positiveness as well as the first relation are trivial. To prove~\eqref{1436}, note that for any integral irreducible polynomial $g(z)$ of degree $n$, the polynomial $z^n g(z^{-1})$ is also irreducible and has the same degree and the same height. Hence for any region $\Omega\subset\mathbb{C}$ it holds
\[
\Psi(Q;\Omega)=\Psi(Q;\Omega^{-1}),
\]
where $\Omega^{-1}$ is defined as $\Omega^{-1}= \left\{z^{-1}\in\mathbb{C} : z \in \Omega \right\}$. Letting $Q$ tend to infinity, we get by applying Theorem~\ref{1844}
\[
\int_{\Omega} \psi(z)\,\nu(dz) =\int_{\Omega^{-1}} \psi(z)\,\nu(dz) .
\]
On the other hand, after the substitution $z\to1/z$, we obtain
\[
\int_{\Omega} \psi(z)\,\nu(dz) =\int_{\Omega^{-1}} \psi(z^{-1})|z|^{-4}\,\nu(dz).
\]
Since the class of regions $\Omega$ is sufficiently large, \eqref{1436} follows.
\end{proof}

\begin{proposition}\label{prop-repuls}
Near the real line the density $\psi$ admits the following asymptotic approximation:
\begin{equation}\label{1541}
\psi(x_0+iy)=A|y|\cdot(1+o(1)),\quad y\to0,
\end{equation}
where the constant $A$ does not depend on $y$ and can be written explicitly as follows:
\begin{equation*}
A = \int\limits_{\widetilde{D}_{x_0}} \left|\sum_{k=1}^{n-1} k(k+1) t_k x_0^{k-1}\right|^2 dt_1\dots dt_{n-1}.
\end{equation*}
Here the integration is performed over the region
\begin{multline*}
\widetilde{D}_{x_0} = \left\{(t_1,\dots,t_{n-1})\in\mathbb{R}^{n-1} : \max\limits_{1\le k\le n-1} |t_k|\le 1, \phantom{\left|\sum_{1}^1\right|} \right. \\
\left. \left|\sum_{k=1}^{n-1} k t_k x^{k+1}_0\right|\le 1, \ \left|\sum_{k=1}^{n-1} (k+1) t_k x^{k}_0\right|\le 1 \right\}.
\end{multline*}
\end{proposition}
Relation~\eqref{1541} may be regarded as a ``repulsion'' of exponent $1$ of complex roots from the real axis.

\begin{proof}
Since
\[
\frac{\Im z^{k+1}}{\Im z} = \frac{z^{k+1} - \bar{z}^{k+1}}{z - \bar{z}} = \sum_{j=0}^{k} z^{k-j}\, \bar{z}^j,
\]
it follows that
\begin{align*}
(k+1) z^{k} - \frac{\Im z^{k+1}}{\Im z} &= \sum_{j=0}^{k} z^{k-j} \left(z^j - \bar{z}^j\right)\\
&=(z - \bar{z}) \sum_{j=1}^{k} z^{k-j} \sum_{m=0}^{j-1} z^{j-1-m} \bar{z}^m = 
(z - \bar{z}) \sum_{s=1}^{k} s z^{s-1} \bar{z}^{k-s}.
\end{align*}
Hence  $\psi(z)$ and $D_z$ can be rewritten as follows:
\begin{equation*}
\psi(z) = 4\, |\Im z|
\int_{D_z} \left|\sum_{k=1}^{n-1} t_k \sum_{s=1}^{k} s z^{s-1} \bar{z}^{k-s}\right|^2 dt_1\dots dt_{n-1},
\end{equation*}
and
\begin{multline*}
D_z = \left\{(t_1,\dots,t_{n-1})\in\mathbb{R}^{n-1} : \max\limits_{1\le k\le n-1} |t_k|\le 1, \phantom{\left|\sum_{1}^1\right|} \right. \\
\left.
\left|\sum_{k=1}^{n-1} t_k \sum_{j=1}^{k} z^{k-j+1}\, \bar{z}^j \right|\le 1, \
\left|\sum_{k=1}^{n-1} t_k \sum_{j=0}^{k} z^{k-j}\, \bar{z}^j\right|\le 1 \right\}.
\end{multline*}
Note that $\widetilde{D}_{x_0} = D_{x_0+0\cdot i}$. Letting $\Im z \to 0$ concludes the proof.
\end{proof}

\begin{proposition}\label{prop-asymp}
For $|z|\ge 1$, the function $\psi(z)$ can be estimated by
\[
\psi(z) \asymp_n \frac{|\Im z|}{|z|^6},
\]
where the implicit constant depends on $n$ only.
\end{proposition}
\begin{proof}
It follows from Proposition~\ref{prop-repuls} that $\psi(z) \asymp_n |\Im z|$ for $|z|\le 1$. Hence~\eqref{1436} yields the proof.
\end{proof}

If $|z|$ is relatively small or relatively large, then it is possible to write the limit density in a simpler form.
\begin{proposition}
If $|z|\leq 1-1/\sqrt2$, then
\[
\psi(z) = \frac{2^{n-1}}{3\,|\Im z|} \sum_{k=1}^{n-1} \left|(k+1) z^{k} - \frac{\Im z^{k+1}}{\Im z}\right|^2.
\]
If $|z|\geq 2+\sqrt2$, then
\[
\psi(z) = \frac{2^{n-1}}{3\,|\Im z|} \sum_{k=1}^{n-1} \frac{1}{|z|^{4k+4}}\left|(k+1) \bar{z}^{k} - \frac{\Im z^{k+1}}{\Im z}\right|^2.
\]
\end{proposition}
\begin{proof}
For $|z|\leq 1-1/\sqrt2$ it holds
\[
\sum_{k=2}^n (k-1) |z|^k \le 1, \text{ and }
\sum_{k=2}^n k |z|^{k-1} \le 1,
\]
which leads to
\[
D_z=[-1,1]^{n-1},
\]
and a straightforward integration yields the first relation. The second statement follows from the first one and \eqref{1436}.
\end{proof}

Let us conclude the section by considering the case $n=2$.

{\bf Example.} In the case of quadratic algebraic numbers the density function takes the form
\[
\psi(z) = \frac{4}{|\Im z|} \int\limits_{D_z} \left|t \Im z\right|^2 dt,
\]
where $$D_z = \left\{t \in\mathbb{R} : |t| \le \min\left(1, \frac1{|z|^{2}}, \frac{1}{2|\Re z|}\right)\right\}.$$

By some elementary transformations, we obtain
\begin{equation*}
\psi(x+iy) =
\begin{cases}
\frac{8}{3} y, & \text{if } x^2 + y^2 \le 1, \text{ and } |x|\le \frac{1}{2},\\
\frac{y}{3 x^3}, & \text{if } (|x|-1)^2 + y^2 \le 1, \text{ and } |x| > \frac{1}{2},\\
\frac{8 y}{3 (x^2 + y^2)^3}, & \text{if } (|x|-1)^2 + y^2 > 1, \text{ and } x^2 + y^2 > 1.
\end{cases}
\end{equation*}

\section{Proof of Theorem~\ref{1844}}\label{1845}
We start with some notation.

For any Borel set $A\subset \mathbb{R}^d$ denote by $\Vol(A)$ the Lebesgue measure of $A$, denote by $\lambda(A)$ the number of points in $A$ with integer coordinates, and denote by $\lambda^*(A)$ the number of points in $A$ with coprime integer coordinates. The Riemann zeta function is denoted by $\zeta(\cdot)$ and the M\"obius function is denoted by $\mu(\cdot)$.

Denote by $\mathcal{P}_Q$ the class of all integral polynomials of degree at most $n$ and height at most $Q$. The cardinality of this class is $(2Q+1)^{n+1}$. Recall that an integral polynomial is called \emph{prime}, if it is irreducible over $\mathbb{Q}$, primitive (the greatest common divisor of its coefficients equals 1), and its leading coefficient is positive.

For $k\in\{0,1,\dots,n\}$ denote by $\gamma_k$ the number of prime polynomials from $\mathcal{P}_Q$ that have exactly $k$ roots lying in $\Omega$. For any algebraic number its minimal polynomial is prime, and any prime polynomial is a minimal polynomial for some algebraic number. Therefore,
\begin{equation}\label{1900}
\Psi(Q;\Omega)= \sum_{k=1}^n k\gamma_k.
\end{equation}
Consider a subset $A_k\subset[-1,1]^{n+1}$ consisting of all points $(t_0,\dots,t_n)\in[-1,1]^{n+1}$ such that the polynomial $t_nx^n+\dots+t_1x+t_0$ has exactly $k$ roots lying in $\Omega$. Then the number of primitive polynomials from $\mathcal{P}_Q$ which have exactly $k$ roots in $\Omega$ is equal to $\lambda^*(QA_k)$. By the definition of a prime polynomial, we have that
\begin{equation}\label{1901}
\left|\gamma_k- \frac12 \lambda^*(QA_k)\right|\leq R_Q,
\end{equation}
where $R_Q$ denotes a number of reducible  polynomials (over $\mathbb Q$) from $\mathcal{P}_Q$. Note that the factor $\frac12$ arises in the above inequality because prime polynomials have positive leading coefficient.
It is known (see~\cite{bW36}) that
\begin{equation}\label{1902}
R_Q=O\left(Q^n\log^{l(n)}Q\right),\quad Q\to\infty.
\end{equation}
There do not exist reducible over $\mathbb{Q}$ integral quadratic polynomials having non-real roots.
Hence it follows from~\eqref{1900}, \eqref{1901}, and~\eqref{1902} that
\begin{equation}\label{245}
\Psi(Q;\Omega)= \frac12 \sum_{k=1}^n k \lambda^*(Q A_k)+O\left(Q^n \right),\quad Q\to\infty.
\end{equation}
To estimate $\lambda^*(QA_k)$, we need the following lemma.

\begin{lemma}\label{1540}
Consider a region $A\subset\mathbb{R}^d$, $d\geq2,$ with boundary consisting of a finite number of algebraic surfaces only. Then
\begin{equation}\label{1318}
\lambda^*(tA)=\frac{\Vol(A)}{\zeta(d)}t^d+O\left(t^{d-1}\log^{l(d)}t\right),\quad t\to\infty.
\end{equation}
Here the implicit constant in the big-O-notation depends on $d$, the number of the algebraic surfaces, and their maximal degree only.
\end{lemma}
The results of this type are well-known, see, e.g., the  classical monograph by Bachmann \cite[pp. 436--444]{pB94} (in particular, formulas~(83a) and~(83b) on pages 441--442). For the readers convenience we include a  short proof here.
\begin{proof}
Note that
\[
\lambda(tA)=\sum_{j=1}^{[Nt]+1} \lambda^*\left(\frac tj A\right),
\]
where $N$ is chosen to be so large that $A\subset [-N,N]^d$. Applying the classical M\"obius inversion formula (see, e.g., \cite{hR77}) yields
\begin{equation}\label{021}
\lambda^*(tA)=\sum_{j=1}^{[Nt]+1}\mu(j)\,\lambda\left(\frac tj A\right).
\end{equation}
By the Lipschitz principle (see~\cite{hD51}) it follows that
\begin{equation}\label{806}
\left|\lambda\left(\frac tj A\right)-\left(\frac{t}{j}\right)^{d}\Vol(A)\right|\leq c \cdot \left(\frac{t}{j}\right)^{d-1}
\end{equation}
for some constant $c$ depending on the number of the algebraic surfaces and their maximal degree only. Applying this to \eqref{021} we get
\begin{equation}\label{807}
\left|\lambda^*(tA)-\Vol(A)t^d \sum_{j=1}^{[Nt]+1}\frac{\mu(j)}{j^d}\right|\leq c\, t^{d-1}\sum_{j=1}^{[Nt]+1}\frac{1}{j^{d-1}}.
\end{equation}
It is well known (see, e.g., \cite{hR77}) that
\[
 \sum_{j=1}^\infty\frac{\mu(j)}{j^d}=\frac{1}{\zeta(d)}.
\]
Therefore,
\begin{equation}\label{808}
 \left|\sum_{j=1}^{[Nt]+1}\frac{\mu(j)}{j^d}-\frac{1}{\zeta(d)}\right|\leq\sum_{j=[Nt]+2}^\infty\frac{1}{j^d}\leq\frac{1}{(d-1)(Nt)^{d-1}}.
\end{equation}
Furthermore, it holds that
\begin{equation}\label{809}
\sum_{j=1}^{[Nt]+1}\frac{1}{j^{d-1}}\leq
\begin{cases}
\zeta(d-1),\quad d\geq3,\\
\log([Nt]+1)+1,\quad d=2.
\end{cases}
\end{equation}
Combining \eqref{807}, \eqref{808}, and \eqref{809} completes the proof.
\end{proof}

The right-hand side of~\eqref{1318} is estimated by the right-hand sides of~\eqref{806} and ~\eqref{808} which are of the same order. The one involving the M\"obius function can be made slightly sharper (by a logarithmic factor) using an unconditional estimate for the Mertens function (see, e.g.,~\cite{oRam13}). Assuming the Riemann hypothesis the latter can be improved more (see \cite{kSou09}). However, the error term in the Lipschitz principle  can be made  smaller for special type of regions only (see~\cite{mSkr98}), which is not our case.

Since the boundary of $\Omega$ consists of a finite number of algebraic curves, the boundary of $A_k$ consists of a finite number of algebraic surfaces. Thus it follows from Lemma~\eqref{1540} that
\[
\lambda^*(QA_k)=\frac{\Vol(A_k)}{\zeta(n+1)}Q^{n+1}+O\left(Q^n\right),\quad t\to\infty,
\]
which together with~\eqref{245} implies
\begin{equation}\label{1140}
\Psi(Q;\Omega)= \frac{Q^{n+1}}{2\zeta(n+1)}\sum_{k=1}^n k\Vol(A_k) +O\left(Q^n \right),\quad Q\to\infty.
\end{equation}

To calculate $\sum_{k=1}^n k\Vol(A_k)$, we need the following result from the theory of random polynomials.
Let $\xi_0,\xi_1,\dots,\xi_{n}$ be independent random variables uniformly distributed on $[-1,1]$. Consider the random polynomial
\[
G(x)=\xi_nx^n+\xi_{n-1}x^{n-1}+\dots+\xi_1x+\xi_0.
\]
Denote by $N(\Omega)$ the number of the roots of $G(z)$ lying in $\Omega$. By definition,
\[
\Vol(A_k) = 2^{n+1}\,\mathbb{P}(N(\Omega)=k),
\]
which implies
\begin{equation}\label{1141}
\sum_{k=1}^n k\Vol(A_k) = 2^{n+1}\, \mathbb{E} N(\Omega).
\end{equation}
The right-hand side of the latter relation was calculated in~\cite{dZ04} in a more general setup: it was shown that if the coefficients $\xi_0,\xi_1,\dots,\xi_n$ have a joint probability density function $p(x_0,x_1,\dots,x_n)$, then $\mathbb{E} N(\Omega)$ is given by the formula
\begin{align}\label{1044}
\mathbb{E} N(\Omega)& = \int\limits_\Omega dr d\alpha \int\limits_{\mathbb{R}^{n-1}} dt_1 \dots dt_{n-1}\, \frac{r^2}{\sin \alpha} \\\notag&\times 
\left(
\left[\sum_{k=1}^{n-1} t_k r^{k-1} \left((k+1) \cos (k+1)\alpha - \cos\alpha\, \frac{\sin (k+1)\alpha}{\sin\alpha} \right)\right]^2
\right. \\\notag&+
\left.
\left[\sum_{k=1}^{n-1} k t_k r^{k-1} \sin (k+1)\alpha \right]^2
\right)\\\notag&\times
p\left(\frac{1}{\sin\alpha} \sum_{k=1}^{n-1} t_k r^{k+1} \sin k\alpha,\ -\frac{1}{\sin\alpha} \sum_{k=1}^{n-1} t_k r^{k} \sin {(k+1)\alpha} , \ t_1,\ \dots,\ t_{n-1} \right),
\end{align}
where $r = |z|$ and $\alpha = \arg z$ are polar coordinates in the complex plane.
The corresponding formula in~\cite{dZ04} contains a typo. Here we use the correct version.

In the case when the coefficients are independent and uniformly distributed on $[-1,1]$, their joint probability density function equals 
$$
p=2^{-n-1}\mathbbm{1}_{[-1,1]^{n+1}}.
$$ 
Thus it follows from~\eqref{1141} and~\eqref{1140} that to finish the proof, it is enough to show that for this specific $p$ the right-hand side of~\eqref{1044} is equal to
\[
\int\limits_\Omega\psi(z)\,\nu(dz),
\]
where $\psi$ is defined in~\eqref{1132}. 

Indeed, the integrand in~\eqref{1132} can be transformed as follows:
\begin{multline*}
\left|\sum_{k=1}^{n-1}t_k\left((k+1)z^{k}-\frac{\Im z^{k+1}}{\Im z}\right)\right|^2 =
\frac{1}{r^2}\left|\sum_{k=1}^{n-1}t_k\left((k+1)z^{k+1}-z\frac{\Im z^{k+1}}{\Im z}\right)\right|^2 \\
= \left|\sum_{k=1}^{n-1} t_k r^{k} \left(\left[(k+1) \cos (k+1)\alpha - \cos\alpha\, \frac{\sin (k+1)\alpha}{\sin\alpha}\right] +
i \left[\vphantom{\frac11} k \sin (k+1)\alpha \right]\right)
\right|^2,
\end{multline*}
and the functions that define the region $D_z$ can be transformed as follows:
\begin{align*}
\Bigg|\sum_{k=1}^{n-1} t_k \bigg(z^{k+1} &- z\frac{\Im z^{k+1}}{\Im z}\bigg)\Bigg| = \left|\sum_{k=1}^{n-1} t_k \left(\Re z^{k+1} - \Re z \frac{\Im z^{k+1}}{\Im z}\right)\right| \\
&=\left|\frac{1}{\sin\alpha} \sum_{k=1}^{n-1} t_k r^{k+1} \left(\sin\alpha \cos(k+1)\alpha - \cos\alpha \sin(k+1)\alpha\right)\right| \\
&=\left|\frac{1}{\sin\alpha} \sum_{k=1}^{n-1} t_k r^{k+1} \sin k\alpha \right|,
\end{align*}
and
\[
\left|\frac{1}{\Im z}\sum_{k=1}^{n-1} t_k \Im z^{k+1}\right| = \left|\frac{1}{\sin\alpha} \sum_{k=1}^{n-1} t_k r^{k} \sin {(k+1)\alpha}\right|.
\]
The proof follows.

{\bf Acknowledgments.}
We wish to thank Vasili Bernik for his valuable comments, Natalia Budarina and Hanna Husakova for helpful discussions, and the anonymous referee for his helpful remarks.

\bibliographystyle{abbrv}
\bibliography{corrf2}

\end{document}